\documentclass [11pt,twoside,a4paper]{article}
\usepackage{mathrsfs}
\usepackage{amssymb}
\usepackage{amsthm}
\usepackage{amsfonts}
\usepackage{amsmath}
\usepackage{amstext}
\usepackage{amscd}
\usepackage{dsfont}
\usepackage{color}
\newtheorem{theorem}{Theorem}[section]

\newtheorem{lemma}{Lemma}[section]

\newtheorem{remark}{Remark}[section]
\newtheorem{definition}{Definition}[section]

\def\BE{\begin{equation}}
\def\EE{\end{equation}}

\newcommand{\bey}{\begin{eqnarray}}
\newcommand{\eey}{\end{eqnarray}}
\newcommand{\beyy}{\begin{eqnarray*}}
\newcommand{\eeyy}{\end{eqnarray*}}

\begin{document}

\title{$\ \ $ On isometric minimal immersion of a singular  non-CSC extremal K$\ddot{a}$hler metric into  3-dimensional space forms}
\author{Zhiqiang Wei\footnotemark[1], Yingyi Wu\footnotemark[2]}

\date{}
%\dedicatory{}%
%\commby{}%
\maketitle
% ----------------------------------------------------------------

\begin{abstract}
On any compact Riemann surface there always exists  a singular  non-CSC (constant scalar curvature) extremal K$\ddot{a}$hler metric which is  called a non-CSC  HCMU (the Hessian of the Curvature of the Metric is Umbilical) metric. In this paper, by moving frames, we show that any  non-CSC HCMU metric can not be isometrically minimal immersed into 3-dimensional real space forms even locally. In general, any non-CSC HCMU metric can not be isometrically immersed into  3-dimensional real space forms with constant mean curvature (CMC).
\end{abstract}

\footnotetext[1]{Supported by Natural Science Foundation of Henan, No.202300410067 an the National Natural Science Foundation of China No.12171140.}
\footnotetext[2]{Supported by the National Natural Science Foundation of China No.11971450.}
% ----------------------------------------------------------------
\section{Introduction}
~~~~~Since Calabi proposed the famous Calabi conjecture, K$\ddot{a}$hler-Einstein metric is one of the hot topics in geometry. For the existence of K$\ddot{a}$hler-Einstein metrics, one can refer to \cite{Yau1,Yau2} \cite{Au}. In 1982, Calabi\cite{Ca} replaced K$\ddot{a}$hler-Einstein metric with extremal K$\ddot{a}$hler metric.  In a fixed K\"{a}hler class, an extremal K\"{a}hler metric is the critical
point of the following Calabi energy functional
$$
\mathcal{C}(g)=\int_{M} R^2 dg,
$$
where $R$ is the scalar curvature of the metric $g$ in the given K$\ddot{a}$hler class. The Euler-Lagrange equations of $\mathcal{C}(g)$ are $R_{,\alpha
\beta}=0$ for all indices $\alpha, \beta$, where $R_{,\alpha\beta}$ is the second-order $(0,2)$
covariant derivative of $R$. When $M$ is a compact Riemann surface,
 Calabi in \cite{Ca} proved that an extremal K$\ddot{a}$hler metric is a CSC (constant scalar curvature) metric. \par

A natural question is whether or not an extremal K$\ddot{a}$hler metric with singularities on a compact Riemann surface is still a CSC metric. In \cite{Ch1}, X.X.Chen  first gave an example of a non-CSC extremal K\"{a}hler metric with singularities. We often call a non-CSC extremal K$\ddot{a}$hler metric with finite singularities on a compact Riemann surface a non-CSC HCMU(the Hessian of the Curvature of the Metric is Umbilical) metric. In \cite{CW2},\cite{Xb}, Q.Chen, B.Xu and Y.Y.Wu reduced the existence of a non-CSC HCMU metric to the existence of a meromorphic 1-form on the underlying Riemann surface. It is interesting that on any compact Riemann surface there always exists a non-CSC HCMU metric. For more properties of non-CSC HCMU metrics, one can refer to \cite{Ch2},\cite{CCW},\cite{CW1},\cite{LZ}, \cite{WZ},\cite{Wei1}  and the references cited in these papers.\par

Recently, isometric immersions of  a non-CSC HCMU metric into some ``good" higher dimensional spaces have been studied. In \cite{WP}, C.K.Peng and Y.Y.Wu proved that any non-CSC HCMU metric can be locally isometric  immersed into 3-dimension Euclidean space $\mathbb{E}^{3}$. They got a one-parameter family of isometric immersions from a compact Riemann surface with a singular non-CSC extremal K$\ddot{a}$hler metric to  $\mathbb{E}^{3}$, each of whom is a Weingarten surface. In \cite{Wei22}, we proved that any non-CSC HCMU metric can be locally isometric immersed into 3-dimensional space forms. As an application, we proved that any non-CSC HCMU metric can be locally isometric immersed into complex projective space $\mathbb{C}P^{n}(n\geq 3)$ with Fubini-Study metric.\par

In this manuscript, we consider the following question: Suppose $g$ is a non-CSC HCMU metric on a compact Riemann surface $M$; For any point $P\in M$, whether or not there exist an open neighborhood $U$ of $P$ and an isometric minimal immersion $F:U \rightarrow \mathbb{Q}^{3}_{c}$, where $\mathbb{Q}^{3}_{c}$ denotes the 3-dimensional space form with section curvature $c$.
The following theorem is our main result.
\begin{theorem}\label{mainth}
Let $g$ be a non-CSC HCMU metric on a compact Riemann surface $M$ with  the character 1-form $\omega$.  Denote $M^{*}=M\setminus\{ \text{zeros and poles of} ~\omega\}$.
Then for any point $P\in M^{*}$, and any open neighborhood $ U\subseteq M^{*}$ of $P$, there doesn't exist an isometric  minimal immersion $F:U\rightarrow \mathbb{Q}^{3}_{c}$.
\end{theorem}
Furthermore, we can prove the following theorem in a similar way.
\begin{theorem}\label{mainth1}
Let $g$ be a non-CSC HCMU metric on a compact Riemann surface $M$ with  the character 1-form $\omega$. Denote $M^{*}=M\setminus\{ \text{zeros and poles of} ~\omega\}$.
Then for any point $P\in M^{*}$, and any open neighborhood $ U\subseteq M^{*}$ of $P$, there doesn't exist an isometric immersion $F:U \rightarrow \mathbb{Q}^{3}_{c}$ of constant mean curvature.
\end{theorem}

\section{Preliminaries}
\subsection{Non-CSC HCMU metric}
\begin{definition}[\cite{T}]\label{conedef}
 Let $M$ be a Riemann surface, $P\in M$. A conformal metric $g$ on $M$ is said to have a conical singularity at
 $P$ with the singular angle $2 \pi \alpha(\alpha>0,\alpha\neq 1)$ if in a neighborhood of $P$
 \begin{equation}\label{ds^2}
  g=e^{2\varphi}|dz|^2,
 \end{equation}
 where $z$ is a local complex coordinate defined in the neighborhood of $P$ with $z(P)=0$ and
 \begin{equation*}
  \varphi-(\alpha-1)\ln|z|
 \end{equation*}
 is continuous at $0$.
\end{definition}
\begin{definition}[\cite{Xb}]\label{cuspdef}
 Let $M$ be a Riemann surface, $P\in M$. A conformal metric $g$ on $M$ is said to have a cusp singularity at
 $P$ if in a neighborhood of $P$
 \begin{equation}\label{cusp-ds^2}
  g=e^{2\varphi}|dz|^2,
 \end{equation}
 where $z$ is a local complex coordinate defined in the neighborhood of $P$ with $z(P)=0$ and
 \begin{equation*}
 \lim_{z\rightarrow 0} \frac{\varphi+\ln|z|}{\ln|z|}=0.
 \end{equation*}
\end{definition}
\begin{definition}[\cite{Ch2}]
 Let $M$ be a compact Riemann surface and $P_1,\cdots,P_N$ be $N$ points on $M$.
 Denote $M\backslash \{P_1,\ldots,P_N\}$ by $M^*$. Let $g$ be a conformal metric on $M^*$.
 If $g$ satisfies
\begin{equation}\label{HCMUequ}
  K_{,zz}=0,
\end{equation}
where $K$ is the  Gauss curvature of $g$, we call $g$ an HCMU metric on $M$.
\end{definition}
  In this paper, we always consider non-CSC HCMU metrics with finite area and
  finite Calabi energy, that is,
\begin{equation}\label{finite}
 \int_{M^*}dg<+\infty, ~~ \int_{M^*}
K^2 dg <+\infty.
\end{equation}
\par
From \cite{Ch0}, \cite{LZ}, \cite{WZ},  we know that each singularity of a non-CSC HCMU metric  is  conical or cusp if it has finite area and finite Calabi energy.\par
We now list some results of non-CSC HCMU metrics, which will be used in this paper. For more results one can refer to \cite{CCW},\cite{Xb} and the references cited in it.\par
First the equation (\ref{HCMUequ}) is equivalent to
\begin{equation*}
 \nabla K =\sqrt{-1}e^{-2\varphi}K_{\bar{z}}\frac{\partial}{\partial z},
\end{equation*}
which is a holomorphic vector field on $M^*$. In \cite{Ch2},\cite{LZ}, the authors proved that the curvature
$K$ can be continuously extended to $M$ and there are finite smooth extremal points of $K$ on
$M^*$. In \cite{CCW},\cite{Xb}, the authors proved the following fact: each smooth extremal
point of $K$ is either the maximum point of $K$ or the minimum point of $K$, and if we denote the maximum
of $K$ by $K_1$ and the minimum of $K$ by $K_2$ then if all the singularities of $g$ are conical singularities,
$$
 K_1>0,~K_1>K_2>-(K_1+K_2);
$$
if there exist cusps in the singularities,
$$K_{1}>0,~K_{2}=-\frac{1}{2}K_{1}.$$
In \cite{LZ}, C.S.Lin and X.H.Zhu proved that $\nabla K$ is actually a meromorphic vector field on $M$.
In \cite{CW2}, Q.Chen and the second author defined the dual 1-form of $\nabla K$ by $\omega(\nabla K)=\frac{\sqrt{-1}}{4}$.
They call $\omega$ the character 1-form of the metric. Denote $M^* \setminus \{\text{smooth extremal
points of}~K \}$ by $M'$. Then on $M'$
\begin{equation}\label{sys0}
\begin{cases}
 \cfrac{dK}{-\frac{1}{3}(K-K_1)(K-K_2)(K+K_1+K_2)}=\omega+\bar{\omega}, \\
g=-\frac{4}{3}(K-K_1)(K-K_2)(K+K_1+K_2)\omega \bar{\omega}.\\
\end{cases}
\end{equation}
By (\ref{sys0}), some properties of $\omega$ are got in \cite{CW2}:
\begin{itemize}
 \item $\omega$ only has simple poles,
 \item at each pole, the residue of $\omega$ is a non-zero real number,
  \item $\omega+\bar{\omega}$ is exact on $M \setminus \{poles~of~\omega\}$.
 \end{itemize}
Conversely, if a meromorphic 1-form $\omega$ on $M$ which satisfies the properties above, then we pick two real numbers $K_{1},K_{2}$ such that $K_{1}>0,K_1>K_2>-(K_1+K_2)$ or $K_{1}>0,~K_{2}=-\frac{1}{2}K_{1},$  and consider the following equation on $M\setminus\{\text{poles of}~\omega\}$
\begin{equation}\label{sys1}
\begin{cases}
 \cfrac{dK}{-\frac{1}{3}(K-K_1)(K-K_2)(K+K_1+K_2)}=\omega+\bar{\omega}, \\
K(P_{0})=K_{0},\\
\end{cases}
\end{equation}
where $P_{0}\in M\setminus\{\text{poles of}~\omega\}$ and $K_{2}< K_{0}< K_{1}$. We get that (\ref{sys1}) has a unique solution $K$ on $M\setminus\{\text{poles of}~\omega\}$ and $K$ can be continuously extended to $M$. Furthermore, we define a metric $g$ on $M\setminus\{\text{poles of}~\omega\}$ by
$$g=-\frac{4}{3}(K-K_{1})(K-K_{2})(K+K_{1}+K_{2})\omega\overline{\omega},$$
 where $K$ is the solution of (\ref{sys1}). Then it can be proved that $g$ is a non-CSC HCMU metric, $K$ is the Gauss curvature of $g$ and $\omega$ is the character 1-form of $g$.\par
It is interesting that on any compact Riemann surface there always exists a meromorphic 1-form  satisfying the properties(see \cite{CCW}). So there always exists a non-CSC HCMU metric on a compact Riemann surface.

\subsection{Riemannian submanifolds}
In this section, we recall some  facts of Riemannian submanifolds. For more results, one may consult \cite{MD} and references cited in it.\par
Let $F:M^{n}\rightarrow \overline{M}^{n+p}$ be an immersion of a smooth manifold $M$ of dimension $n$ into a smooth manifold $\overline{M}$ of dimension $n+p$. The number $p$ is called the codimension of $F$. If $<,>_{\overline{M}}$ is a Riemannian metric on $\overline{M}$, for every point $P\in M$ and any $X,Y\in T_{P}M$, define $<X,Y>_{M}=<F_{*}X,F_{*}Y>_{\overline{M}}$. Then $<,>_{M}$ is a Riemannian metric on $M$. In this case, $F$ becomes an isometric immersion of $M$ into $\overline{M}$. We will often drop the subscript and denote a Riemannian metric simply by $<,>$, assuming that the underlying manifold will be clear from the context.\par

Let $F:M^{n}\rightarrow \overline{M}^{n+p}$ be an isomeric immersion. Since $F$ is an immersion, then, for each point $P\in M$, there exists a neighborhood $U\subseteq M$ of $P$ such that $F:U\rightarrow \overline{M}$ is an imbedding. Therefore, we may identity $U$ with $F(U)$. Hence, the tangent space of $M$ at $P$ is a subspace of $\overline{M}$ at $P$. Then we have
\begin{equation}\label{De-E}
T_{P}\overline{M}=T_{P}M\oplus T^{\perp}_{P}M,
\end{equation}
where $T^{\perp}_{P}M$ is the orthogonal complement of $T_{P}M$ in $T_{P}\overline{M}$. In this way, we obtain a vector bundle
$$T^{\perp}M=\bigcup_{P\in M}T_{P}^{\perp}M,$$
which is called the normal bundle of $M$.\par

Let $\nabla,\overline{\nabla}$ be the Levi-Civita connections of $M,\overline{M}$, respectively.  Denote the sets of smooth vector fields and smooth normal vector fields on $M$ by $\chi(M),\chi^{\bot}(M)$, respectively. Then for any two smooth vector fields $X,Y\in\chi(M)$, by (\ref{De-E}), we obtain the Gauss formula
$$\overline{\nabla}_{X}Y=\nabla_{X}Y+B(X,Y),$$
where $B:TM\times TM\rightarrow T^{\perp}M$ is called the second fundamental form of $F$.\par
Similarly, for any $X\in\chi(M),\xi\in\chi^{\bot}(M)$, by (\ref{De-E}), we obtain the Weingarten formula
$$\overline{\nabla}_{X}\xi=-A_{\xi}X+\nabla^{\bot}_{X}\xi,$$
where $A_{\xi}:TM\rightarrow TM$ is called the shape operator of $f$ with respect to $\xi$, and $\nabla^{\perp}$ is called the normal connection of $F$. By the Gauss and Weingaren formulas, $B$ and $A_{\xi}$ satisfy
\begin{equation}\label{A-B}
<A_{\xi}X,Y>=<B(X,Y),\xi>.
\end{equation}

If the codimension $p = 1$, we call the isometric immersion $F:M^{n}\rightarrow \overline{M}^{n+1}$ is a hypersurface of $\overline{M}$. Let $F:M^{n}\rightarrow \overline{M}^{n+1}$ be an orientable hypersurface. Choosing a
local smooth unit normal vector field $\xi$ along $F$ and a local smooth orthonormal  tangential frame $e_{1},\ldots,e_{n}$, then the mean curvature vector $H$ of $F$  is defined by
$$H=\frac{1}{n}\sum_{i=1}^{n}B(e_{i},e_{i}).$$
Denote $A = A_{\xi}$, then, by (\ref{A-B}),
$$H=\frac{1}{n}(\sum_{i=1}^{n}<Ae_{i},e_{i}>)\xi.$$
If $H\equiv0$, the isometric immersion $F$ is called a minimal immersion. Generally, $F$ is called a constant mean curvature immersion if $\|H\|$ is a constant.\par

\subsubsection{Basic equations}
Using the Gauss and Weingarten formulas, the basic equations of isometric immersion $F:M^{n}\rightarrow \overline{M}^{n+p}$ can be written  as follows.
\begin{equation*}
\begin{aligned}
&\textbf{Gauss-equation}\\
R(X,Y,Z,W)=\overline{R}(X,Y,Z,W)+&<B(X,Z),B(Y,W)>-<B(X,W),B(Y,Z)>;
\end{aligned}
\end{equation*}
\begin{equation*}
\begin{aligned}
&\textbf{Codazzi-equation}\\
(\overline{R}(X,Y)Z)^{\perp}=&(\nabla^{\perp}_{X}B)(Y,Z)-(\nabla^{\perp}_{Y}B)(X,Z);
\end{aligned}
\end{equation*}
\begin{equation*}
\begin{aligned}
&\textbf{Ricci-equation}\\
(\overline{R}(X,Y)\xi)^{\perp}=R^{\perp}(&X,Y)\xi+B(A_{\xi}X,Y)-B(X,A_{\xi}Y),
\end{aligned}
\end{equation*}
where $X,Y,Z,W\in \chi(M),\xi\in \chi^{\perp}(M)$, $R^{\perp}$ denotes the curvature tensor of the normal bundle $T^{\perp}M$ and $R,\overline{R}$ are Riemannian curvature tensors of $M,\overline{M}$, respectively.\par
In particular, if $\overline{K}(X,Y)=\overline{R}(X,Y,X,Y)$ and $K(X,Y)=R(X,Y,X,Y)$ denote the sectional curvatures in $\overline{M}$ and $M$ of the plane generated by the orthonormal vectors $X,Y\in T_{P}M$, the Gauss-equation becomes
\begin{equation*}
K(X,Y)=\overline{K}(X,Y)+<B(X,X),B(Y,Y)>-<B(X,Y),B(X,Y)>.
\end{equation*}

In the case of a hypersurface $F:M^{n}\rightarrow \overline{M}^{n+1}$, the Gauss-equation can be written as
$$R(X,Y,Z,W)=\overline{R}(X,Y,Z,W)-<AX,W><AY,Z>+<AX,Z><AY,W>.$$
The Codazzi-equation becomes
$$(\overline{R}(X,Y)\xi)^{T}=(\nabla_{Y}A)(X)-(\nabla_{X}A)Y,$$
where
$$(\nabla_{Y}A)X=\nabla_{Y}AX-A\nabla_{Y}X.$$

Moreover, if $\overline{M}^{n+1}$ has constant section curvature c, then the basic equations reduce, respectively, to\par
\textbf{Gauss-equation}\\
$$R(X,Y)Z=c(X\wedge Y)Z+(AX\wedge AY)Z,$$
where $(X\wedge Y)Z=<Y,Z>X-<X,Z>Y$.\par

\textbf{Codazzi-equation}\\
$$(\nabla_{Y}A)X=(\nabla_{X}A)Y.$$
\begin{remark}
In the case of hypersurfaces, the Ricci-equation is identity.
\end{remark}
We now, using moving frames, give the basic equations of the hypersurface $F:M^{n}\rightarrow \overline{M}^{n+1}$. We will make use of the following convention on the ranges of indices:
$$1\leq A,B,C,\ldots \leq n+1,$$
$$1\leq i,j,k,\ldots\leq n,$$
and we shall agree that repeated indices are summed over the respective.\par
Let $e_{1},\ldots,e_{n},e_{n+1}$ be a  local orthonormal frame of $\overline{M}$, such that $e_{1},\ldots,e_{n}$ are tangential to $M$, then $e_{n+1}$ is perpendicular to $M$. Let $\theta^{1},\ldots,\theta^{n},\theta^{n+1}$ be its dual coframe. Then the structure equations of $\overline{M}$ can be written as follows:
\begin{equation*}
\begin{cases}
d\theta^{A}=-\theta^{A}_{B}\wedge \theta^{B},\theta^{A}_{B}+\theta^{B}_{A}=0,\\
d\theta^{A}_{B}=-\theta^{A}_{C}\wedge\theta^{C}_{B}+\Phi^{A}_{B},\Phi^{A}_{B}=\frac{1}{2}\overline{R}^{A}_{BCD}\theta^{C}\wedge\theta^{D},
\end{cases}
\end{equation*}
where $\theta^{A}_{B}$ and $\Phi^{A}_{B}$ are connection forms and curvature forms of $\overline{M}$.\par
Set
$$F^{*}\theta^{A}=\omega^{A},F^{*}\theta^{A}_{B}=\omega^{A}_{B},$$
then the structure equations of $M$ are
\begin{equation*}
\begin{cases}
d\omega^{i}=-\omega^{i}_{j}\wedge \omega^{j},\omega^{i}_{j}+\omega^{j}_{i}=0,\\
d\omega^{i}_{j}=-\omega^{i}_{k}\wedge\omega^{k}_{j}+\Omega^{i}_{j},\Omega^{i}_{j}=\frac{1}{2}R^{i}_{jkl}\omega^{k}\wedge\omega^{l}.
\end{cases}
\end{equation*}
The basic equations are
\begin{equation*}
(\textbf{Gauss-equation})~~R^{i}_{jkl}=\overline{R}^{i}_{jkl}+(h^{n+1}_{ik}h^{n+1}_{jl}-h^{n+1}_{il}h^{n+1}_{jk}),
\end{equation*}
\begin{equation*}
(\textbf{Codazzi-equation})~~\overline{R}^{n+1}_{ijk}=h^{n+1}_{ikj}-h^{n+1}_{ijk},
\end{equation*}
where $\omega^{n+1}_{i}=h^{n+1}_{ij}\omega^{j},h^{n+1}_{ij}=h^{n+1}_{ji},h^{n+1}_{ijk}\omega^{k}=dh^{n+1}_{ij}-h^{n+1}_{ik}\omega^{k}_{j}-h^{n+1}_{kj}\omega^{k}_{i}.$ In fact, by (\ref{A-B}), we have
\begin{equation*}
\begin{aligned}
A(e_{i})=\sum_{j=1}^{n}h^{n+1}_{ij}e_{j}.
\end{aligned}
\end{equation*}
If the section curvature of $\overline{M}$ is a constant $c$, then the basic equations become
\begin{equation}\label{G-C-E}
\begin{cases}
R_{ijkl}=c(\delta_{ik}\delta_{jl}-\delta_{il}\delta_{jk})+h_{ik}h_{jl}-h_{il}h_{jk}~(\textbf{Gauss-equation}),\\
h_{ikj}=h_{ijk}~(\textbf{Codazzi-equation}),
\end{cases}
\end{equation}
where $h_{ij}=h^{n+1}_{ij}.$

\subsubsection{The Fundamental Theorem  of Hypersurfaces}
From now on, let $\overline{M}^{n+1}=\mathbb{Q}_{c}^{n+1}$, where $\mathbb{Q}_{c}^{n+1}$ denotes $(n+1)$-dimension space form with constant sectional curvature $c$. Then  the fundamental theorem of hypersurfaces can be written as follows.
\begin{theorem}[\cite{MD}]\label{F-T-H}
Let $M^{n}$ be a simply connected Riemannian manifold,  and let $A$ be a symmetric section of $End(TM)$ satisfying the Gauss and Codazzi equations. Then there exist an isometric immersion $F:M^{n}\rightarrow \mathbb{Q}^{n+1}_{c}$ and a unit normal vector field $\xi$ such that $A$ coincides with the shape operator $A_{\xi}$ of $F$ with respect to $\xi$.
\end{theorem}

\section{Proof of Theorem \ref{mainth}}
\subsection{Reduce the existence of the isometric minimal immersion $F:U\rightarrow \mathbb{Q}^{3}_{c}$ to the existence of some kind of  1-forms}
By the theorem \ref{F-T-H} and the basic equations (\ref{G-C-E}), one can easily prove the following lemma.
\begin{lemma}\label{L-1}
Let $M$ be a simply connected Riemann surface. Let $g=(\omega^{1})^{2}+(\omega^{2})^{2}$ be a Riemannian metric of $M$, and $\omega^{2}_{1}$ be the connection form of $g$, then there exists an isometric minimal immersion $F:M\rightarrow \mathbb{Q}^{3}_{c}$ if and only if there exist two 1-forms
\begin{equation*}
\begin{cases}
\omega_{1}^{3}=h_{11}\omega^{1}+h_{12}\omega^{2},\\
\omega_{2}^{3}=h_{21}\omega^{1}+h_{22}\omega^{2},
\end{cases}
\end{equation*}
which satisfy
\begin{equation*}
\begin{cases}
h_{11}=-h_{22},\\
h_{12}=h_{21},
\end{cases}
\end{equation*}
and
\begin{equation*}
\begin{cases}
d\omega_{1}^{2}=-\omega_{1}^{3}\wedge \omega_{2}^{3}-c\omega^{1}\wedge\omega^{2}~(\textbf{Gauss-equation}),\\
d\omega_{1}^{3}=\omega_{1}^{2}\wedge \omega_{2}^{3}~(\textbf{Codazzi-equation}),\\
d\omega_{2}^{3}=-\omega_{1}^{2}\wedge \omega_{1}^{3}~(\textbf{Codazzi-equation}).
\end{cases}
\end{equation*}
\end{lemma}
\subsection{Proof of Theorem \ref{mainth}}
\begin{lemma}\label{L-2}
Let $M$ be a compact Riemann surface, and $g$ be a non-CSC HCMU metric on $M$. Suppose $\omega$ and $K$ are the character 1-form  and the Gauss curvature of $g$. Suppose the maximum and the minimum of $K$ are $K_{1},K_{2}$ respectively. Denote $M\setminus\{\text{zeros and poles of }\omega\}$ by $M^{*}$, $\sqrt{-\frac{4}{3}(K-K_{1})(K-K_{2})(K+K_{1}+K_{2})}$ by $\mu=\mu(K)$. If for any point $P\in M^{*}$, there exist an open neighborhood $P\in U\subseteq M^{*}$ and an isometric minimal immersion $F:U\rightarrow \mathbb{Q}^{3}_{c}$, then there is a complex value function $h$ such that
\begin{equation*}
\begin{cases}
K=c-\frac{4|h|^{2}}{\mu^{2}},\\
b=-\frac{\mu'\mu h}{4},\\
a=\frac{3\mu'\mu h}{4}+\frac{\mu^{2}h}{4(K-c)},
\end{cases}
\end{equation*}
where $dh=a\omega+b\overline{\omega}$.
\end{lemma}
\begin{proof}
Set
\begin{equation*}
\begin{cases}
\omega^{1}=\frac{\omega+\overline{\omega}}{2}\mu,\\
\omega^{2}=\frac{\omega-\overline{\omega}}{2\sqrt{-1}}\mu.
\end{cases}
\end{equation*}
Then, by (\ref{sys0}),
\begin{equation*}
g=\mu^{2}\omega\overline{\omega}=(\omega^{1})^{2}+(\omega^{2})^{2},
\end{equation*}
and
\begin{equation*}
dK=\frac{\mu^{2}}{4}(\omega+\overline{\omega}).
\end{equation*}
Since
$$d\omega^{1}=\mu'(K)dK\wedge\frac{\omega+\overline{\omega}}{2}=0,$$
$$d\omega^{2}=\mu'dK\wedge\frac{\omega-\overline{\omega}}{2\sqrt{-1}}=\frac{\mu'}{2}\omega^{1}\wedge\omega^{2},$$
then the connection 1-form of $g$ is
\begin{equation*}
\omega^{2}_{1}=\frac{\mu'}{2}\omega^{2}.
\end{equation*}
By Lemma \ref{L-1}, there exist two 1-forms
\begin{equation*}
\begin{cases}
\omega_{1}^{3}=h_{11}\omega^{1}+h_{12}\omega^{2},\\
\omega_{2}^{3}=h_{21}\omega^{1}+h_{22}\omega^{2},
\end{cases}
\end{equation*}
satisfying
\begin{equation}\label{M-C}
\begin{cases}
h_{11}=-h_{22},\\
h_{12}=h_{21},
\end{cases}
\end{equation}
and
\begin{equation*}
\begin{cases}
d\omega_{1}^{2}=-\omega_{1}^{3}\wedge \omega_{2}^{3}-c\omega^{1}\wedge\omega^{2}~(\textbf{Gauss-equation}),\\
d\omega_{1}^{3}=\omega_{1}^{2}\wedge \omega_{2}^{3}~(\textbf{Codazzi-equation}),\\
d\omega_{2}^{3}=-\omega_{1}^{2}\wedge \omega_{1}^{3}~(\textbf{Codazzi-equation}).
\end{cases}
\end{equation*}
Assume
\begin{equation*}
\begin{cases}
\omega_{1}^{3}=f\omega+\overline{f}\overline{\omega},\\
\omega_{2}^{3}=h\omega+\overline{h}\overline{\omega}.
\end{cases}
\end{equation*}
Then
$$h_{11}=\frac{f+\overline{f}}{\mu},h_{12}=\frac{\sqrt{-1}(f-\overline{f})}{\mu},h_{21}=\frac{h+\overline{h}}{\mu},h_{22}=\frac{\sqrt{-1}(h-\overline{h})}{\mu}.$$
So, by (\ref{M-C}),
$$f=-\sqrt{-1}h.$$
Therefore,
\begin{equation*}
\begin{cases}
\omega_{1}^{3}=-\sqrt{-1}(h\omega-\overline{h}\overline{\omega}),\\
\omega_{2}^{3}=h\omega+\overline{h}\overline{\omega}.
\end{cases}
\end{equation*}
Since
\begin{equation*}
\begin{cases}
d\omega_{1}^{2}=-K\omega^{1}\wedge\omega^{2},\\
\omega_{1}^{3}\wedge\omega_{2}^{3}=\frac{-4|h|^{2}}{\mu^{2}}\omega^{1}\wedge\omega^{2},
\end{cases}
\end{equation*}
then the Gauss-equation becomes
$$K=c-\frac{4|h|^{2}}{\mu^{2}}.$$
Let $dh=a\omega+b\overline{\omega}$, then $d\overline{h}=\overline{a}~\overline{\omega}+\overline{b}\omega$, and
\begin{equation*}
\begin{cases}
d\omega_{1}^{3}=\sqrt{-1}(b+\overline{b})\omega\wedge\overline{\omega},\\
d\omega_{2}^{3}=(\overline{b}-b)\omega\wedge\overline{\omega}.
\end{cases}
\end{equation*}
Since
\begin{equation*}
\begin{cases}
\omega_{1}^{2}\wedge\omega_{2}^{3}=\frac{\mu'\mu}{4\sqrt{-1}}(h+\overline{h})\omega\wedge\overline{\omega},\\
\omega_{1}^{2}\wedge\omega_{1}^{3}=\frac{-\mu'\mu}{4}(h-\overline{h})\omega\wedge\overline{\omega},
\end{cases}
\end{equation*}
then the Codazzi-equation becomes
\begin{equation*}
\begin{cases}
b+\overline{b}=\frac{-\mu'\mu}{4}(h+\overline{h}),\\
b-\overline{b}=\frac{-\mu'\mu}{4}(h-\overline{h}),
\end{cases}
\end{equation*}
i.e.,
\begin{equation*}
b=-\frac{\mu'\mu}{4}h.
\end{equation*}
To sum up, we get
\begin{equation*}
\begin{cases}
K=c-\frac{4|h|^{2}}{\mu^{2}}~~(\textbf{Gauss-equation}),\\
b=-\frac{\mu'\mu h}{4}~~(\textbf{Codazzi-equation}).
\end{cases}
\end{equation*}
Differentiating two sides of the Gauss-equation, we get
$$a\overline{h}+\overline{b}h=-\frac{\mu^{2}[2\mu'\mu (K-c)+\mu^{2}]}{16}.$$
Since $b=-\frac{\mu'\mu}{4}h$, then
$$a=-\frac{\mu^{2}[3\mu'\mu (K-c)+\mu^{2}]}{16\overline{h}}=\frac{3\mu'\mu h}{4}+\frac{\mu^{2}h}{4(K-c)}.$$
\end{proof}

\begin{lemma}\label{L-3}
There does not exist a function $h$ satisfying the conditions in Lemma \ref{L-2}.
\end{lemma}
\begin{proof}
Since $dh=a\omega+b\overline{w}$, so
$$d^{2}h=d(a\omega+b\overline{\omega})=da\wedge\omega+db\wedge \overline{\omega}=0.$$
Since
\begin{equation*}
\begin{aligned}
&da\equiv \frac{\mu^{3}h}{16}[3\mu''+\frac{\mu'}{K-c}-\frac{\mu}{(K-c)^{2}}]\overline{\omega}~~(mod ~\omega),\\
&db\equiv-\frac{\mu^{2}h}{16}[\mu''\mu+4(\mu')^{2}+\frac{\mu'\mu}{K-c}]\omega~~(mod~\overline{\omega}),\\
&da\wedge\omega=\frac{\mu^{3}h}{16}[3\mu''+\frac{\mu'}{K-c}-\frac{\mu}{(K-c)^{2}}]\overline{\omega}\wedge\omega,\\
&db\wedge\overline{\omega}=\frac{\mu^{2}h}{16}[\mu''\mu+4(\mu')^{2}+\frac{\mu'\mu}{K-c}]\overline{\omega}\wedge\omega,
\end{aligned}
\end{equation*}
so
$$da\wedge\omega+db\wedge \overline{\omega}=\frac{\mu^{2}h}{16(K-c)^{2}}[4\mu''\mu (K-c)^{2}+4(\mu')^{2}(K-c)^{2}+2\mu'\mu (K-c)-\mu^{2}]\overline{\omega}\wedge\omega=0.$$
Thus
\begin{equation}\label{L-E}
4\mu''\mu (K-c)^{2}+4(\mu')^{2}(K-c)^{2}+2\mu'\mu (K-c)-\mu^{2}=0.
\end{equation}
Suppose
$$\mu=\sqrt{-\frac{4}{3}(K-K_{1})(K-K_{2})(K+K_{1}+K_{2})}=(-\frac{4}{3}K^{3}+\lambda_{1}K+\lambda_{2})^{1/2},$$
where $\lambda_{1}=\frac{3}{4}(K_{1}^{2}+K_{2}^{2}+K_{1}K_{2}),\lambda_{2}=\frac{3}{4}K_{1}K_{2}(K_{1}+K_{2})$,
then
$$\mu'=\frac{1}{2}(-\frac{4}{3}K^{3}+\lambda_{1}K+\lambda_{2})^{-1/2}(-4K^{2}+\lambda_{1}),$$
$$\mu''=-\frac{1}{4}(-\frac{4}{3}K^{3}+\lambda_{1}K+\lambda_{2})^{-3/2}(-4K^{2}+\lambda_{1})^{2}-4K(-\frac{4}{3}K^{3}+\lambda_{1}K+\lambda_{2})^{-1/2},$$
$$\mu\mu''=-\frac{1}{4}(-\frac{4}{3}K^{3}+\lambda_{1}K+\lambda_{2})^{-1}(-4K^{2}+\lambda_{1})^{2}-4K,$$
$$(\mu')^{2}=\frac{1}{4}(-\frac{4}{3}K^{3}+\lambda_{1}K+\lambda_{2})^{-1}(-4K^{2}+\lambda_{1})^{2},$$
$$\mu\mu'=\frac{1}{2}(-4K^{2}+\lambda_{1}),$$
$$\mu\mu''+(\mu')^{2}=-4K,$$
So
$$
4\mu''\mu (K-c)^{2}+4(\mu')^{2}(K-c)^{2}\neq\mu^{2}-2\mu'\mu (K-c),
$$
that is the identity (\ref{L-E}) is not true.
\end{proof}
The proof of Theorem \ref{mainth} obtains from  Lemmas \ref{L-2},\ref{L-3}.\par

% ----------------------------------------------------------------

\par\vskip0.3cm
\noindent
Zhiqiang Wei\\
School of Mathematical and Statistics \\
Henan University \\
Kaifeng 475004 P.R. China \\
weizhiqiang15@mails.ucas.ac.cn\\
\vskip0.3cm
 \noindent
 Yingyi Wu \\
School of Mathematical Sciences \\
University of Chinese Academy of Sciences \\
Beijing 100049 P.R. China \\
wuyy@ucas.ac.cn
\end{document}